\documentclass[12pt,a4paper]{amsart}
\usepackage{amssymb}
\usepackage{amsmath,amsthm}
\usepackage{a4wide}
\newtheorem{theorem}{Theorem}[section]
\newtheorem{lemma}[theorem]{Lemma}
\newtheorem{proposition}[theorem]{Proposition}
\usepackage{cite}

\theoremstyle{definition}

\def\Z{\mathbb Z}

\def\N{\mathbb N}
\def\S{\mathcal{S}}

\newcommand{\R}{\mathbb{R}}

\theoremstyle{remark}
\newtheorem{remark}[theorem]{Remark}

\numberwithin{equation}{section}

\theoremstyle{remark}

\usepackage{color}

\usepackage{xcolor}

\begin{document}

\title[Asymptotic results for  FRST and FRWT]{Asymptotic results for the distributional fractional Stockwell and fractional wavelet transforms}

\author[S. Maksimovi\'{c}]{Snje\v{z}ana Maksimovi\'{c}}
\address{Faculty of Architecture, Civil Engineering and Geodesy, University of Banja Luka\\ Bulevar vojvode Petra Bojovi\'{c}a 1a \\ 78000 Banja Luka\\ Bosnia and Herzegovina}
\email{snjezana.maksimovic@aggf.unibl.org}
%
%


\maketitle
\begin{abstract}
We provide some Abelian and Tauberian results characterizing the quasiasymptotic behavior of Lizorkin distributions  in terms of their Stockwell transform. We prove the continuity of the  fractional wavelet transform and the corresponding
synthesis operator on the Schwartz  spaces and their duals, respectively. Additionally, we establish a connection between fractional Stockwell and fractional wavelet transforms and provide some asymptotic results for the distributional fractional wavelet transform.
\end{abstract}
\vspace{0.2cm}




\section{Introduction}
\label{Sec:1}
It is well known that distributions do not have point values in the general case, which imposes the idea of absorbing the asymptotic analysis into the field of generalized functions.
The natural generalization of classical asymptotics in the framework of Schwartz distributions is the quasiasymptotic behavior (quasiasymptotics) introduced in \cite{VDZ}.
It is an old mathematical tool that has found applications  in  various  fields  of  pure  and  applied  mathematics,  physics,  and  engineering. Many results involve the analysis of the behavior of distributions through various integral transforms and Abelian and Tauberian theory (cf. \cite{6-2,  KPSV, Vindas4, 7, 8,  VPR}).

In the last two decades, fractional transforms have played an important role in several areas, including signal and image processing, optics, geo-informatics, radar signal and communications, bio-medical, etc. (cf. \cite{Alm, AMP,Capus, Namias,Zayed, Pathak,  shi}). 
The fractional Fourier transform (FRFT), first introduced in \cite{Namias}, is a generalization of the Fourier transform (FT) in which the FT kernel is substituted with the fractional kernel $K_\alpha$, i.e.,
\begin{equation*}\label{frstft}\mathcal{F}_\alpha f(\xi)=
\int_{\R}f(x)K_{\alpha}(x,\xi)dx, \quad \xi\in \R,
  \end{equation*}
where
\begin{equation}\label{kernel} K_{\alpha}(x,\xi)=\begin{cases}
C_\alpha e^{i(\frac{x^2+\xi^2}{2}c_1-x\xi c_2)}, &\text{ $\alpha\neq n\pi$}\\
 \ \ \ \delta(x-\xi), &\text{ $\alpha=2n\pi$} \\
     \ \ \ \delta(x+\xi), &\text{ $\alpha=(2n+1)\pi$}\\

  \end{cases}  
\end{equation}
 $c_1=\cot\alpha, \ c_2=\csc \alpha$, $C_\alpha=\sqrt{\frac{1-ic_1}{2\pi}}$, $n\in\N$.  The kernel $K_{\alpha}(x,\xi)$ is an infinitely differentiable function in both $x$ and $\xi$.
 The fractional short-time Fourier transform (FRSTFT), as a generalization of the short-time Fourier transform (STFT), is proposed to solve the problem of the FRFT when failing to locate the fractional Fourier domain frequency contents (cf. \cite{Capus}). The fractional wavelet transform (FRWT), first introduced in \cite{mendlo} (see also \cite{shi}), is a generalization of the wavelet transform (WT). This transform is proposed to rectify the limitations of the WT and the FRFT. The generalization of the Stockwell transform (ST) is the fractional Stockwell transform (FRST). The FRST is first introduced in \cite{s1} to deal with seismic data and improve the deficiencies of the STFT and WT (cf. \cite{Alm,s1}). Since then, the FRST has played an important role in signal and image processing, and has become the most popular multiresolution tool (cf. \cite{s5, maks1, s4}).

In this paper, we focus on the FRST and FRWT introduced in \cite{shi}. The paper is divided into two parts. The first part (Section \ref{se1}) is a continuation of the work by \cite{maks}, where several Abelian and Tauberian-type results connecting the quasiasymptotic behavior of Lizorkin distributions to the asymptotic behavior of their FRST are provided. In the second part (Section \ref{se2}), we establish the relation between the distributional FRST and FRWT. Additionally, continuity theorems for the FRWT (Theorems \ref{wte1} and \ref{wte2}) are given. By using these continuity theorems, we demonstrate that both transforms, defined as transposed mappings and as actions on appropriate window functions, coincide on the Schwartz space of tempered distributions. Moreover, we present some results regarding the asymptotic behavior of distributional FRWT, utilizing the obtained asymptotic results for the FRST.

\section{Preliminaries}
\subsection{Notation and spaces}
We employ the notation   $A\lesssim_{\alpha} B$,  which means that $A\leq D(\alpha)\cdot B$ for some  positive constant $D$ that depends on a parameter $\alpha$. For the measurable function $f$ on $\R$, by $M_a f(\cdot)=e^{ia\cdot}f(\cdot)$ we denote the modulation operator and for $\varepsilon>0$ we set $f_\varepsilon(\cdot)=f(\varepsilon\cdot)$. The FT of the function $f$ is defined as: $\hat f(\xi)=\mathcal F(f)(\xi)=\frac{1}{\sqrt{2\pi}}\int_{\R}f(x)e^{- ix\xi}dx$, and inverse FT is $\mathcal F^{-1}(f)(\xi)=\frac{1}{\sqrt{2\pi}}\int_{\R}f(x)e^{ix\xi}dx$.  For the $L^2$ inner product of $f$ and $g$ we use $(f,g)$, and for the dual pairing we use $\langle f,g\rangle$.

The linear space of all rapidly decreasing smooth functions at infinity $\mathcal{S}(\R)$ consists of all functions $f\in\mathcal{C}^\infty(\R)$ for which
$
\rho_{k,p}(f)=\sup_{x\in\R}|x^kf^{(p)}(x)|<\infty, \ k,p\in\N_0.$
The dual space of $\mathcal{S}(\R)$ is $\mathcal{S}'(\R)$ and it is called the Schwartz space of tempered distributions (cf. \cite{Schwartz}). 

The space of highly localized test functions $\mathcal{S}_0(\R)$  is defined as the closed subspace of the space
$\mathcal{S}_0(\R)=\{\varphi\in\mathcal{S}(\R)\mid \int _{\R}x^k\varphi(x)dx=0, \ k\in\N_0\}$.
 Its dual space, denoted by $\mathcal{S}'_0(\mathbb{R})$, is called the Lizorkin distribution space (cf. \cite{wt}).

Let $\mathbb Y=\R\times\R\setminus\{0\}$. We recall the definition of  the space $\mathcal{S}(\mathbb Y)$, which consists of all functions $f\in\mathcal{C}^\infty(\mathbb{Y})$ for which
$
\rho^{l,m}_{s,r}(f)=\sup_{(x,\xi)\in\mathbb{Y}}(|\xi|^s+\frac{1}{|\xi|^s})|x^r\partial_\xi^l\partial_x^m f(x,\xi)|<\infty
$
for all $m,l,s,r\in\N_0$ (cf. \cite{wt}). 

Recall  \cite{Schwartz, wtp} that the space $\mathcal{S}(\mathbb R\times\R^+)$ consists of all functions $f\in\mathcal{C}^\infty(\mathbb R\times\R^+)$ for which
$
\sigma^{l,m}_{s,r}(f)=\sup_{(x,\xi)\in\mathbb R\times\R^+}\xi^s|x^r\partial_\xi^l\partial_x^m f(x,\xi)|<\infty $
for all $m,l,s,r\in\N_0$.

\subsection{Quasiasymptotics}
Recall \cite{7}  that a measurable real valued function, defined and positive on an
interval $(0,A]$ (resp. $[A,\infty )),$ $A>0$, is called {a
slowly varying function} at the origin (resp. at infinity) if
\begin{equation} \label{limL1}
\lim_{\varepsilon \to 0^{+} } \frac{L(a\varepsilon
)}{L(\varepsilon )} =1\quad (\ {\rm resp.} \lim_{\lambda \to
\infty } \frac{L(a\lambda )}{L(\lambda )} =1)\quad {\rm for \ \
each} \ a>0.
\end{equation}

Let $L$ be a slowly varying function at the origin. Recall \cite{KSB} that
the distribution $f\in {\mathcal S_0'}({\Bbb R})$ has
{quasiasymptotic behavior} ({quasiasymptotics}) of
degree $m \in {\Bbb R}$ at the point $x_{0} \in {\Bbb R}$
with respect to $L$, if there exists $u\in {\mathcal S_0'}({\Bbb R})$
such that for each $\varphi \in {\mathcal S_0}({\Bbb R})$
\begin{equation} \label{qbeh}
\lim_{\varepsilon \to 0^{+} } \langle \frac{f(x_{0} +\varepsilon
x)}{\varepsilon ^{m } L(\varepsilon )} ,\ \varphi (x)\rangle
=\langle u(x),\varphi (x)\rangle .
\end{equation}
We will use the following convenient notation for the
quasiasymptotic behavior:
\[f(x_{0} +\varepsilon x) \sim \varepsilon ^{m } L(\varepsilon )u(x)\quad {\rm as} \quad \varepsilon \to 0^{+} \quad {\rm in} \quad {\mathcal S_0'}({\Bbb R}),\]
which should always be interpreted in the weak topology of
$\mathcal S'_0({\Bbb R})$, i.e., in the sense of \eqref{qbeh}.

It is proven in  \cite{PST,
VDZ} that the limit distribution $u$ has to be homogeneous with a degree of homogeneity $m $, i.e.,
$u(ax)=a^{m } u(x)$, for all $a>0 $. We remark that all homogeneous distributions on the real
line are explicitly known; indeed, they are linear combinations of
either $x_{+}^{m} $ and $x_{-}^{m } $, if $m
\notin {\Bbb Z}_{-} $, or $\delta ^{(k-1)} (x)$ and $x^{-k} $, if
$m =-k\in {\Bbb Z}_{-} $. It is shown in \cite{Vindas2} that  the
quasiasymptotic behavior at finite points is a local property. The
quasiasymptotics of distributions at infinity with respect to the
slowly varying function $L$ at infinity is defined in a similar
way, and in this case we use the notation $f(\lambda x) \sim \lambda ^{m }
L(\lambda )u(x)$ as $\lambda \to \infty $ in ${\mathcal S_0'}({\Bbb
R})$ (cf. \cite{Vindas1}). It is shown in \cite{AMP} that the following assertion holds:
\begin{lemma}\label{l1}
 If
\begin{equation}\label{123} \langle f(\varepsilon x)/(\varepsilon^m L(\varepsilon)), \psi(x)\rangle,  \mbox{ converges   as} \ \ \varepsilon\rightarrow 0^+, \,\,\forall \psi\in\mathcal S,
\end{equation}
then
\begin{equation}\label{123a} \langle e^{ic(\varepsilon x)^2/2}f(\varepsilon x)/(\varepsilon^m L(\varepsilon)), \psi(x)\rangle,  \,\, \mbox{ converges   as} \ \ \varepsilon\rightarrow 0^+,
\forall \psi\in\mathcal S,
\end{equation}
where $c$ is a real constant and $\varepsilon\in(0,1)$.
Conversely,
if (\ref{123a}) holds and for some $\varepsilon_0\in(0,\varepsilon),$ the family 
$
\{f(\varepsilon x)/(\varepsilon^m L(\varepsilon)):\varepsilon\in(0,\varepsilon_0)\}\mbox{ is bounded in } \mathcal S'(\R),
$
then $(\ref{123a})\; \Rightarrow \; (\ref{123}).$
\end{lemma}
 \section{Abelian and Tauberian results for the fractional Stockwell transform}\label{se1}
 \subsection{The FRST}
  Let $g\in L^1(\R)\cap L^2(\R)$ such that $\int_{\R} g(x)dx=1$. The FRST of the signal $f\in L^2(\R)$ is defined by (cf. \cite{maks})
\begin{equation*}\label{frst}
S^{\alpha}_gf(x,\xi)=|\xi|\int_{\R}f(t)\overline{g(\xi(t-x))}K_{\alpha}(t,\xi)dt,\quad x\in\R,\xi\in\R\setminus{\{0\}},
\end{equation*}
where  $K_{\alpha}(t,\xi)$ is the fractional kernel defined by \eqref{kernel}.
We  shall always assume that $\alpha\in[0, 2\pi)$.  If $\alpha=\pi/2$, the FRST corresponds to
the ST (cf. \cite{KSB})
$$S_gf(x,\xi)=\frac{|\xi|}{\sqrt{2\pi}}\int_{\R}f(t)\overline{g(\xi(t-x))}e^{-i\xi t}dt,\quad x\in\R,\xi\in\R\setminus{\{0\}},$$ 
and the FRST with $\alpha = 0$ corresponds to the
zero operator.

Let $g\in \mathcal{S}(\R)$ be a non-trivial and let $\psi\in\mathcal{S}(\R)$ be a reconstruction
window for $g$. If $f\in L^1(\R)$ is such that $\hat f\in L^1(\R)$, then the following reconstruction formula
holds pointwise
\begin{equation}\label{reconstructionf}
f(t)=\frac{|\sin\alpha|}{C_{g,\psi,c_2}}\int_{\R}\int_{\R}S_g^\alpha f(x,\xi)\psi(\xi(t-x))K_{-\alpha}(t,\xi)dxd\xi,
\end{equation}
where $C_{g,\psi, c_2}=\int_{\R}\hat \psi(c_2(\omega-1))\overline{\hat g}(c_2(\omega-1))\frac{d\omega}{|\omega|}<\infty$ (cf. \cite{maks}). Using the reconstruction formula \eqref{reconstructionf},   the fractional Stockwell synthesis operator was introduced in \cite{maks} as follows:
\begin{equation*}\label{ssynt}
(S_g^{\alpha})^\ast F(t)=|\sin\alpha|\int_{\R}\int_{\R}F(x,\xi)g(\xi(t-x))K_{-\alpha}(t,\xi)dxd\xi.
\end{equation*}
Thus, the relation \eqref{reconstructionf} takes the form
$((S_\psi^\alpha)^*\circ S_g^\alpha)f=C_{g,\psi,c_2}f.$

It is proven in \cite{maks} that the mappings $S_g^\alpha:\mathcal{S}_0(\R)\rightarrow \mathcal{S}(\mathbb Y)$ and $(S_g^\alpha)^*:\mathcal{S}(\mathbb Y)\rightarrow \mathcal{S}_0(\R)$ are continuous, for $g\in\S_0(\R)$. Additionally, the mappings $S_g^\alpha:\mathcal{S}'_0(\R)\rightarrow \mathcal{S}'(\mathbb Y)$ and $(S_g^\alpha)^*:\mathcal{S}'(\mathbb Y)\rightarrow \mathcal{S}_0(\R')$ are continuous, and the reconstruction formula \eqref{reconstructionf} is generalized to the Lizorkin distributions, as presented in \cite{maks}.

In \cite{maks},  several Abelian and Tauberian-type results characterizing the quasiasymptotic behavior of Lizorkin distributions in terms of their FRST are proven. In order to obtain the quasiasymptotic behavior of the distributional FRWT at the origin, we use the following result from \cite{maks}: If $f\in\S'_0(\R)$ has  the quasiasymptotic behavior at the origin, then for its FRST holds
\begin{equation}\label{rez1}
e^{-ic_1 (\frac{\xi}{\varepsilon})^2/2} S_{g}^\alpha f(\varepsilon x,\frac{\xi}{\varepsilon})\sim \frac{{\sqrt{1-ic_1}}}{c_2^{m}}{\varepsilon^{m}L(\varepsilon)} S_{g}u(x c_2, \frac{\xi}{c_2}) \ \ \text{ as } \ \ \varepsilon\rightarrow0^+ \ \ \text{ in } \ \mathcal{S}_0'(\mathbb Y).
\end{equation}

\subsection{Abelian and Tauberian results for the FRST} 
Our main goal in this section is to present some  new Abelian and Tauberian-type results relating asymptotics of FRST and the quasiasymptotic
behavior of Lizorkin distributions.

\begin{theorem}\label{teab1}
Let $f\in \mathcal{S}_0'(\R)$ has the following quasiasymptotic behavior
\begin{equation}\label{newq}
f(\varepsilon x)\sim {\varepsilon^mL(\varepsilon)} u(x) \ \ \text{ as } \ \ \varepsilon\rightarrow0^+ \ \ \text{ in } \ \mathcal{S}_0'(\R),
\end{equation}
where $m\in\mathbb R$, and $L$ is a slowly varying function at 0 and $u\in\mathcal S_0'(\mathbb R)$. Then, for its FRST with respect to the window $g\in  \mathcal{S}_0(\R)\setminus\{0\}$, we have
\[e^{-ic_1 ({\varepsilon}\xi)^2/2} S_{g_{1/\varepsilon^2}}^\alpha f(\varepsilon x,{\varepsilon}\xi)\sim \frac{\sqrt{1-ic_1}}{c_2^{m}}{\varepsilon^{m+2}L(\varepsilon)} S_{g}(M_{\xi/c_2} u)(x c_2, \frac{\xi}{c_2}) \ \ \text{ as } \ \ \varepsilon\rightarrow0^+ \ \ \text{ in } \ \mathcal{S}_0'(\mathbb Y),\]
whenever $\alpha\in(2k\pi,(2k+1)\pi), k\in\mathbb Z$.
\end{theorem}
\begin{proof}
Let $(x,\xi)\in\mathbb Y$ be fixed. Then, we have

\begin{align*}
 e^{-ic_1 ({\varepsilon}\xi)^2/2} S_{g_{1/\varepsilon^2}}^\alpha f(\varepsilon x,{\varepsilon}\xi)&= e^{-ic_1 ({\varepsilon}\xi)^2/2}C_\alpha{|\varepsilon\xi|} \langle f(t),{\overline {g}\big(\frac{\xi(t-\varepsilon x)}{\varepsilon}\big)e^{ic_1\frac{t^2+(\varepsilon\xi)^2}{2}-ic_2t\varepsilon\xi}}\rangle\\
&=C_\alpha{|\varepsilon^2\xi|} \langle e^{ic_1(\varepsilon t)^2/2}f(\varepsilon t),{\overline {g}\big({\xi(t- x)}\big)e^{-ic_2t\varepsilon^2\xi}}\rangle\\
&=\frac{C_\alpha}{c_2}{|\varepsilon^2\xi|} \langle e^{ic_1(\frac{\varepsilon t}{c_2})^2/2}f(\frac{\varepsilon t}{c_2}),{\overline {g}\big({\xi(\frac{ t}{c_2}- x)}\big)e^{-it\varepsilon^2\xi}}\rangle
\end{align*}
and
\begin{align*}
\lim_{\varepsilon\rightarrow0^+} \frac{e^{-ic_1 ({\varepsilon}\xi)^2/2} S_{g_{1/\varepsilon^2}}^\alpha f(\varepsilon x,{\varepsilon}\xi)}{\varepsilon^{m+2}L(\varepsilon)}&=
\frac{C_\alpha}{c_2}{|\xi|} \lim_{\varepsilon\rightarrow0^+} \frac{\langle e^{ic_1(\frac{\varepsilon t}{c_2})^2/2}f(\frac{\varepsilon t}{c_2}),{\overline {g}\big({\xi(\frac{ t}{c_2}- x)}\big)e^{-it\varepsilon^2\xi}}\rangle}{\varepsilon^mL(\varepsilon)}\\
&=\frac{C_\alpha}{c^{m+1}_2}{|\xi|} \lim_{(\varepsilon_1,\varepsilon_2)\rightarrow(0^+,0^+)} \frac{\langle e^{ic_1(\frac{\varepsilon_1 t}{c_2})^2/2}f(\frac{\varepsilon_1 t}{c_2}),{\overline {g}\big({\xi(\frac{ t}{c_2}- x)}\big)e^{-it\varepsilon^2_2\xi}}\rangle}{(\frac{\varepsilon_1}{c_2})^mL(\frac{\varepsilon_1}{c_2})\frac{L(\varepsilon_1)}{L(\frac{\varepsilon_1}{c_2})}}.
\end{align*}
Since the weak and the strong topology on $\S_0(\R)$ are equivalent, using \eqref{limL1} and Lemma \ref{l1}, we have
$$ \lim_{\varepsilon_1\rightarrow 0^+} \frac{\langle e^{ic_1(\frac{\varepsilon_1 t}{c_2})^2/2}f(\frac{\varepsilon_1 t}{c_2}),{\overline {g}\big({\xi(\frac{ t}{c_2}- x)}\big)e^{-it\varepsilon^2_2\xi}}\rangle}{(\frac{\varepsilon_1}{c_2})^mL(\frac{\varepsilon_1}{c_2})\frac{L(\varepsilon_1)}{L(\frac{\varepsilon_1}{c_2})}}=\langle u(t),{\overline {g}\big({\xi(\frac{ t}{c_2}- x)}\big)e^{-it\varepsilon^2_2\xi}}\rangle,$$ 
uniformly for $\varepsilon_2\in(0,1]$. Furthermore, for each $\varepsilon_1\in(0,1]$, we have
$$ \lim_{\varepsilon_2\rightarrow 0^+} \frac{\langle e^{ic_1(\frac{\varepsilon_1 t}{c_2})^2/2}f(\frac{\varepsilon_1 t}{c_2}),{\overline {g}\big({\xi(\frac{ t}{c_2}- x)}\big)e^{-it\varepsilon^2_2\xi}}\rangle}{(\frac{\varepsilon_1}{c_2})^mL(\frac{\varepsilon_1}{c_2})\frac{L(\varepsilon_1)}{L(\frac{\varepsilon_1}{c_2})}}=\frac{\langle e^{ic_1(\frac{\varepsilon_1 t}{c_2})^2/2}f(\frac{\varepsilon_1 t}{c_2}),{\overline {g}\big({\xi(\frac{ t}{c_2}- x)}\big)}\rangle}{(\frac{\varepsilon_1}{c_2})^mL(\frac{\varepsilon_1}{c_2})\frac{L(\varepsilon_1)}{L(\frac{\varepsilon_1}{c_2})}}.$$
Thus
\begin{align*}
\lim_{\varepsilon\rightarrow0^+} \frac{e^{-ic_1 ({\varepsilon}\xi)^2/2} S_{g_{1/\varepsilon^2}}^\alpha f(\varepsilon x,{\varepsilon}\xi)}{\varepsilon^{m+2}L(\varepsilon)}&=\frac{C_\alpha}{c^{m+1}_2}{|\xi|} \langle u(t),{\overline {g}\big({\xi(\frac{ t}{c_2}- x)}\big)}\rangle\\
&=\frac{\sqrt{1-ic_1}}{c^{m}_2}S_g(M_{\xi/c_2} u)(c_2 x, \frac{\xi}{c_2}).
\end{align*}
\end{proof}

\begin{theorem}\label {te1} Let $L$
be a slowly varying function at the origin, $m \in {\Bbb R}$
and $f\in {\mathcal S}_0'(\Bbb R)$, $g\in {\mathcal S}_0({\Bbb
R})\backslash\{ 0\} $. Moreover, assume that:

\begin{enumerate}
\item [(i)] the limit
$
\displaystyle\lim_{\varepsilon \to 0^{+} }
\frac{e^{-ic_1 (\varepsilon \xi)^2/2} S_g^\alpha f( \varepsilon  x,\varepsilon\xi)}{\varepsilon ^{m } L(\varepsilon
)}<\infty$ exists for each $(x,\xi)\in\mathbb Y$, and

\item [(ii)] there exist $r\in\N, s>1$ and $0<\varepsilon_0\leq 1$  such that
\begin{equation}\label{ogranicuvanje11a}
\frac{|e^{-ic_1 (\varepsilon \xi)^2/2} S_g^\alpha f(\varepsilon x,\varepsilon \xi)|}{\varepsilon ^{m } L(\varepsilon
)})\lesssim_\alpha \left(|\xi|+\frac{1}{|\xi|}\right)^{-s}|x|^r,
\end{equation}

\end{enumerate}
\noindent for all $(x,\xi)\in \Bbb Y$, $\alpha\in(2k\pi,(2k+1)\pi)$, $k\in\Z$ and $0< \varepsilon \leq
\varepsilon_0$. Then $f$ has a quasiasymptotics at $0^+$.

\end{theorem}

\begin{proof}

Let $\varphi\in\S_0(\R)$. By  the inversion formula  \eqref{reconstructionf} we obtain
\begin{align*}
&\mathop{{\rm lim}}\limits_{\varepsilon \to 0^{+} } \langle \frac{e^{i{c_1(\varepsilon t)^2}/{2}} f(\varepsilon t)}{\varepsilon ^{m } L(\varepsilon )} ,\varphi (t)\rangle =\mathop{{\rm lim}}\limits_{\varepsilon \to 0^{+} }\langle \frac{e^{ic_1 t^2/{2}} f(t)}{\varepsilon ^{m +1} L(\varepsilon )} ,\varphi (\frac{t}{\varepsilon})\rangle\\
&= \mathop{{\rm lim}}\limits_{\varepsilon \to 0^{+} } \langle\frac{e^{ic_1 t^2/{2}}|\sin\alpha|}{C_{g,\psi,c_2}\varepsilon ^{m+1} L(\varepsilon)}\iint_{\mathbb R^2}S_{g}^{\alpha}f(x,\xi)\psi(\xi(t-x))K_{-\alpha}(t,\xi)dxd\xi, \varphi(\frac{t}{\varepsilon})\rangle\\
&= \mathop{{\rm lim}}\limits_{\varepsilon \to 0^{+} } \langle\frac{|\sin\alpha|}{C_{g,\psi,c_2}\varepsilon ^{m-1} L(\varepsilon)}\iint_{\mathbb R^2}e^{-ic_1(\varepsilon \xi)^2/2}S_{g}^{\alpha}f(x\varepsilon,\varepsilon \xi)\psi(\varepsilon \xi(t-x\varepsilon)) e^{it\varepsilon \xi c_2}dxd\xi, \varphi(\frac{t}{\varepsilon})\rangle\\
&= \mathop{{\rm lim}}\limits_{\varepsilon \to 0^{+} } \frac{|\sin\alpha|}{C_{g,\psi,c_2}\varepsilon ^{m-1} L(\varepsilon)}\iint_{\mathbb R^2}e^{-ic_1(\varepsilon \xi)^2/2}S_{g}^{\alpha}f(x\varepsilon,\varepsilon \xi)\left(\int _{\R}\varphi(\frac{t}{\varepsilon}) \psi(\varepsilon \xi(t-x\varepsilon)) e^{it\varepsilon \xi c_2}dt\right)dxd\xi \\
&= \mathop{{\rm lim}}\limits_{\varepsilon \to 0^{+} } \frac{|\sin\alpha|}{C_{g,\psi,c_2}\varepsilon ^{m-2} L(\varepsilon)}\iint_{\mathbb R^2}e^{-ic_1(\varepsilon \xi)^2/2}S_{g}^{\alpha}f(x\varepsilon,\varepsilon \xi)\left(\int _{\R}\varphi(t) \psi(\varepsilon^2\xi(t-x)) e^{it\xi\varepsilon^2 c_2}dt\right)dxd\xi \\
&= \mathop{{\rm lim}}\limits_{\varepsilon \to 0^{+} } \frac{|\sin\alpha|\sqrt{2\pi}}{c_2C_{g,\psi,c_2}\varepsilon ^{m} L(\varepsilon)}\iint_{\mathbb R^2}e^{-ic_1(\varepsilon \xi)^2/2}S_{g}^{\alpha}f(x\varepsilon,\varepsilon \xi)\overline{S_{\psi_{1/c_2}}\overline{\varphi}(x,c_2\varepsilon^2\xi)}dx\frac{d\xi}{|\xi|},
\end{align*}

\noindent where $\psi$ is the synthesis window for $g$. Since $\overline{S_{\psi_{1/c_2}}\overline{\varphi}}\in\S(\mathbb Y)$, by \eqref{ogranicuvanje11a}, we have

$$\left|\frac{e^{-ic_1(\varepsilon \xi)^2/2}S_{g}^{\alpha}f(x\varepsilon,\varepsilon \xi)\overline{S_{\psi_{1/c_2}}\overline{\varphi}(x,c_2\varepsilon^2\xi)}}{\varepsilon ^{m } L(\varepsilon
)}\right|\lesssim_\alpha \left(|\xi|+\frac{1}{|\xi|}\right)^{-s}|x|^{-2}\lesssim_\alpha\frac{1}{(1+|\xi|^2)(1+|x|^2)}$$
for all $(x,\xi)\in\mathbb{Y}$ and $0<\varepsilon\leq \varepsilon_0$.
Therefore, according to the Lebesque’s dominated convergence theorem, we obtain
 that the
limit $\mathop{\lim }\nolimits_{\varepsilon \to 0^{+} } \langle
\frac{e^{i{c_1(\varepsilon t)^2}/{2}}f(\varepsilon t)}{\varepsilon ^{m } L(\varepsilon )}
,\varphi (t)\rangle $  exists for each $\varphi \in {\mathcal
S}_0(\Bbb R)$. So, we conclude that $f$ has quasiasymptotic behavior
at the origin in ${\mathcal S}_0'(\Bbb R)$.

\end{proof}

\begin{remark} 
Similar assertions to the previous ones hold for the quasiasymptotics at infinity.
\end{remark}
\section{Fractional wavelet transform and quasyasimptotic of distributions}\label{se2}
\subsection{Continuity theorems for the FRWT of distributions}
Let $g\in L^2(\R)$ be a wavelet, i.e., let it satisfy $\int_\R g(x)dx=0$. 
The FRWT of the signal ${\displaystyle f\in L^{2}(\mathbb {R} )}$ with respect to the wavelet $g$ is defined as (cf. \cite{shi})
\begin{equation*}\label{talasic}
W_{g}^{\alpha }f(x,\xi)=\frac{1}{\sqrt{\xi}}\int \limits _{\mathbb {R} }f(t)\overline g\left(\frac{t-x}{\xi}\right)e^{ic_1\frac{t^2-x^2}{2}}dt, \quad x\in\R, \ \xi\in\R^+.
\end{equation*}
If $\alpha=\pi/2$, the FRWT corresponds to
the WT
$$W_{g}f(x,\xi)=\frac{1}{\sqrt{\xi}}\int \limits _{\mathbb {R} }f(t)\overline g\left(\frac{t-x}{\xi}\right)dt, \quad x\in\R, \ \xi\in\R^+.$$
The FRWT can be expressed in terms of the FRFT, i.e., 
\begin{equation}\label{talasic1}
W_{g}^{\alpha }f(x,\xi)=\sqrt{2\pi \xi}\int_\R\mathcal{F}_\alpha f(t)\hat g(c_1t\xi)K_{-\alpha}(t,x)dt, \quad x\in\R, \ \xi\in\R^+.
\end{equation}
Let $g\in L^2(\R)$ be a mother wavelet. If $f\in L^2(\R)$, then the following reconstruction formula
holds pointwise (cf. \cite{shi})
\begin{equation}\label{reconstructionfw}
f(t)=\frac{1}{2\pi C_{g}}\iint_{\R^2}W_g^\alpha f(x,\xi)g\left(\frac{t-x}{\xi}\right)e^{-ic_1\frac{t^2-x^2}{2}}\frac{dxd\xi}{\xi^2},
\end{equation}
where $C_{g}=\int_{\R}|{\hat g}(\omega)|^2\frac{d\omega}{|\omega|}<\infty.$ From the reconstruction formula \eqref{reconstructionfw}, the  fractional wavelet synthesis operator is defined by
\begin{equation*}\label{synthesis}
(W_g^\alpha)F(t)=\frac{1}{2\pi}\iint_{\R^2}F(x,\xi)g\left(\frac{t-x}{\xi}\right)e^{-ic_1\frac{t^2-x^2}{2}}\frac{dxd\xi}{\xi^2}.
\end{equation*}

Now, we  extend the results of \cite{shi} for the  Schwartz  spaces   and their duals.
\begin{theorem}\label{wte1}
The mapping $W_g^\alpha:\mathcal S(\R)\times\mathcal{S}(\R)\rightarrow\mathcal{S}(\R\times\R^+)$ is continuous.
\end{theorem}
\begin{proof} Since the FRFT is a continuous mapping from $\S(\R)$ onto itself (cf. \cite{Pathak}), using \eqref{talasic1}, we have that the following mappings are continuous:
$$f\mapsto \mathcal{F}_\alpha f(t)\mapsto h(t,\xi)= \mathcal{F}_\alpha f(t)\hat g(c_1 t\xi)\mapsto \sqrt{2\pi\xi}\int_{\R}h(t,\xi)K_{-\alpha}(t,x)dt, \ \ \xi\in\R^+$$
$$\mathcal{S}(\R)\mapsto\mathcal{S}(\R)\mapsto\mathcal{S}(\R\times\R^+)\mapsto\mathcal{S}(\R\times\R^+).$$
\end{proof}
\begin{theorem}\label{wte2}
The mapping $(W_g^\alpha)^*:\mathcal{S}(\R\times\R^+)\rightarrow\mathcal S(\R)\times\mathcal{S}(\R)$ is continuous.
\end{theorem}
\begin{proof}
Let $g\in\S(\R)$ and  $F\in\S(\mathbb{R}\times \R^+)$. We will show that for  given $k,n\in\N_0$ there exist $a, b, m, p, s,r\in\N_0$ such that
\[\rho_{k,n}((S_g^{\alpha})^{\ast})\lesssim_\alpha \rho_{a,b}(g)\sigma^{m,p}_{s,r}(F).\]
Notice that the $n^{th}$ derivative of the function $\xi\mapsto  e^{i\frac{\xi^2c_1}{2}}$ {has the form $e^{i\frac{\xi^2c_1}{2}}P(\xi)$, where $P(\xi)$} is a polynomial with the leading term $i^nc_1^n\xi^n$.   So, it holds that
$
\big|\partial^n_\xi\big( e^{i\frac{\xi^2c_1}{2}}\big)\big|\leq nC\big|c_1^n\xi^ne^{i\frac{\xi^2c_1}{2}}\big|,
$
where $C$ is a positive constant. Then, we have
\begin{align*}
 &|t^k\partial_t^n ((W^{\alpha}_{g})^{\ast}F(t))|\\
 &\lesssim\sum_{n_1+n_2=n}\binom{n}{n_1,n_2}| (c_1)^{n_2} |
\int_\R\int_{\R^+}| t^{k+n_2}\xi^{-n_1-2}F(x,\xi)g^{(n_1)}(\frac{t-x}{\xi})e^{-ic_1\frac{t^2-x^2}{2}}dx d\xi|\\
& \lesssim_\alpha\sum_{n_1+n_2=n}\binom{n}{n_1,n_2}\int_\R\int_{\R^+}| \xi^{k+n_2-n_1-2}x^{k+n_2}F(x,\xi)\left(\frac{t-x}{\xi}\right)^{k+n_2}g^{(n_1)}(\frac{t-x}{\xi})e^{-ic_1\frac{t^2-x^2}{2}}dx d\xi|\\
&\lesssim_\alpha\sum_{n_1+n_2=n}\binom{n}{n_1,n_2}\sigma_{k,k+n_2}^{0,0}(F)\rho_{k+n_2,n_1}(g),
\end{align*}
from which we obtain the proof.
\end{proof}
By the use of Theorems \ref{wte1} and \ref{wte2}, we can define the FRWT of $f\in \S^\prime(\R)$ with respect to $g\in\S(\R)$ as the element $W^{\alpha}_{g}f\in \S^\prime(\mathbb R\times\R^+)$  whose action on test functions is given by:
\begin{equation}\label{anpair}
\langle W^{\alpha}_{g} f,\Phi\rangle=\langle f, (W_{\overline g}^{\alpha})^\ast\Phi\rangle, \qquad \Phi\in \S(\mathbb R\times\R^+).
\end{equation}
Then, the fractional wavelet synthesis operator $(W_g^{\alpha})^*:\S^\prime(\mathbb R\times\R^+)\rightarrow \S^\prime(\R)$ can be defined as
\begin{equation}\label{sypair}
 \langle (W_g^{\alpha})^* F, \phi\rangle=\langle F, W^{\alpha}_{\overline g}\phi\rangle, \quad F\in\S^\prime(\mathbb R\times\R^+), \phi \in \S(\R).
 \end{equation}

We obtain immediately:
\begin{theorem}  Let $g\in\S(\mathbb R)$. The FRWT $W^{\alpha}_{g}:\S'(\R)\to \S'(\mathbb R\times\R^+)$ and the  synthesis operator $(W_g^{\alpha})^{*}: \S'(\mathbb R\times\R^+)\to\S'(\R)$ are  continuous
linear maps. \end{theorem}
Now, we can {extend} the reconstruction formula \eqref{reconstructionfw} to the space of tempered distributions.
\begin{theorem}\label{ivw}(Inversion formula) Let  $g\in\S(\mathbb R)$, then
\[ \textbf{Id}_{\S'(\R)}=\frac{1}{2\pi C_{g}}((W_{g}^{\alpha})^\ast\circ W^{\alpha}_{g}).\]

\end{theorem}
\begin{proof}
Using \eqref{anpair}, \eqref{sypair} and reconstruction formula \eqref{reconstructionfw}, we obtain
\[\langle (W_{g}^{\alpha})^\ast\circ(W_g^{\alpha} f),\phi\rangle=\langle f,(W_{\overline g}^{\alpha})^\ast\circ(W_{\overline g}^{\alpha} \phi)\rangle=2\pi C_{g}\langle f, \phi\rangle,\quad \phi\in\S(\R).\]
\end{proof}

According to the Hahn-Banach theorem,  the wavelet transform
can be extended over $\mathcal{S}_0'$.
\subsection{Abelian  results for the FRWT}

In this section, we present some Abelian-type results for the FRWT, using the asymptotic results for the FRST. First, we will connect the FRST and FRWT. For ${\displaystyle f\in L^{2}(\mathbb {R} )}$ and $g\in L^{1}(\mathbb {R} )\cap L^{2}(\mathbb {R} )$, by the definitions of FRWT and FRST, one can easily obtain the following relation:
 \begin{equation}\label{veza}
 e^{-ic_1\frac{\xi^2}{2}}S^\alpha_g f(x,\xi)=\sqrt{\xi}{C_\alpha}e^{ic_1\frac{x^2}{2}-ic_2x\xi}W_{M_{c_2}g}^\alpha f(x,\frac{1}{\xi}), \quad x\in\R, \xi\in\R^+.
 \end{equation}
 \begin{lemma}\label{lema1}
 For $g\in\mathcal{S}_0(\R)$, the function $M_{c_2}g$ is also wavelet in $\mathcal{S}_0(\R)$.
 \end{lemma}
 \begin{proof}
 The proof follows from \cite[Lemma 5.1]{buralieva}
 \end{proof}
We want to prove that the relation \eqref{veza} also holds for distributional FRST and FRWT.
Using \eqref{frst}, \eqref{talasic} and Lemma \ref{lema1},  we obtain:
\begin{proposition}\label{prop1}
Let $f\in \S'_0(\R)$ and $g\in\S_0(\R)$, then \eqref{veza} holds.
\end{proposition}

\begin{theorem}\label {te3}
Let $f\in \mathcal{S}_0'(\R)$ has the quasiasymptotic behavior \eqref{newq}. Then, for its FRWT with respect to the window $g\in  \mathcal{S}_0(\R)\setminus\{0\}$, we have
\[{e^{ic_1 (\varepsilon x)^2/2} W_{M_{c_2}g}^\alpha f(\varepsilon x,\frac{\varepsilon}{\xi})}\sim \frac{\varepsilon^{m+1/2}}{c_2^{m+1/2}}L(\varepsilon)e^{ix\xi(c_2-1)}W_{M_{1}g}u(x c_2, \frac{c_2}{\xi}) \ \ \text{ as } \ \ \varepsilon\rightarrow0^+ \]
$\text{ in } \ \mathcal{S}_0'(\R\times\R^+)$ whenever $\alpha\in(2k\pi,(2k+1/2)\pi), k\in\mathbb Z$.
\end{theorem}
\begin{proof}
Using Proposition \ref{prop1} and \eqref{rez1}, we have
\begin{align*}
e^{ic_1 (\varepsilon x)^2/2} W_{M_{c_2}g}^\alpha f(\varepsilon x,\frac{\varepsilon}{\xi})&=\sqrt{\frac{\varepsilon}{\xi}}\frac{1}{C_\alpha}e^{ic_2x\xi}e^{-ic_1\left(\frac{\xi}{\varepsilon}\right)^2/2}S^\alpha_g f(x\varepsilon,\frac{\xi}{\varepsilon})\\
&\sim \sqrt{2\pi}\sqrt{\frac{\varepsilon}{\xi}}\frac{\varepsilon^m}{c_2^m}L(\varepsilon)e^{ic_2 x\xi}S_gu(xc_2,\frac{\xi}{c_2}) \text{ as } \varepsilon\rightarrow0^+.
\end{align*}
Since
\begin{align*}
S_gu(xc_2,\frac{\xi}{c_2})&=\left|\frac{\xi}{c_2}\right|\frac{e^{-ix\xi}}{\sqrt{2\pi}}\int_{\R}u(t)\overline{e^{i\frac{\xi}{c_2}(t-xc_2)}g(\frac{\xi}{c_2}(t-xc_2))}dt\\
&=\sqrt{\frac{\xi}{c_2}}\frac{e^{-ix\xi}}{\sqrt{2\pi}}W_{M_1g}u(xc_2,\frac{c_2}{\xi}),
\end{align*}
we obtain
$$e^{ic_1 (\varepsilon x)^2/2} W_{M_{c_2}g}^\alpha f(\varepsilon x,\frac{\varepsilon}{\xi})\sim \frac{\varepsilon^{m+1/2}}{c_2^{m+1/2}}L(\varepsilon)e^{i(c_2-1) x\xi}W_{M_1g}(xc_2,\frac{c_2}{\xi}), \ \ \text{ as } \  \varepsilon\rightarrow0^+.$$
\end{proof}
\begin{theorem}\label {te4}
Let $f\in \mathcal{S}_0'(\R)$ has the  quasiasymptotic behavior \eqref{newq}. Then, for its FRWT with respect to the window $g\in  \mathcal{S}_0(\R)\setminus\{0\}$, we have
\[{e^{ic_1 (\varepsilon x)^2/2-ic_2\varepsilon^2 x\xi} W_{M_{c_2}g_{1/\varepsilon^2}}^\alpha f(\varepsilon x,\frac{1}{\varepsilon\xi})}\sim \frac{\varepsilon^{m+3/2}}{c_2^{m+1/2}}L(\varepsilon)e^{ix\xi(c_2-1)}W_{g}u(x c_2, \frac{c_2}{\xi}) \ \ \text{ as } \ \ \varepsilon\rightarrow0^+ \]
$\text{ in } \ \mathcal{S}_0'(\R\times\R^+)$ whenever $\alpha\in(2k\pi,(2k+1/2)\pi), k\in\mathbb Z$.
\end{theorem}
\begin{proof}
Using Proposition \ref{prop1}  and Theorem \ref{teab1}, we have
\begin{align*}
e^{ic_1 (\varepsilon x)^2/2-ic_2\varepsilon^2 x\xi} W_{M_{c_2}g_{1/\varepsilon^2}}^\alpha f(\varepsilon x,\frac{1}{\varepsilon\xi})&=\frac{1}{C_\alpha\sqrt{\varepsilon\xi}}e^{-ic_1(\varepsilon\xi)^2/2}S^\alpha_{g_{1/\varepsilon^2}}f(\varepsilon x,\varepsilon\xi)\\
&\sim\sqrt{2\pi}\frac{\varepsilon^{m+2}}{c_2^m\sqrt{\varepsilon\xi}}L(\varepsilon)S_g(M_{\xi/c_2}u)(xc_2,\frac{\xi}{c_2})\\
&=\frac{\varepsilon^{m+3/2}}{c_2^{m+1/2}}L(\varepsilon)W_gf(xc_2,\frac{c_2}{\xi}).
\end{align*}
\end{proof}
\begin{theorem}\label {te5}
Let $f\in \mathcal{S}_0'(\R)$ has the quasiasymptotic behavior \eqref{newq}. Then, for its FRST with respect to the window $g\in  \mathcal{S}_0(\R)\setminus\{0\}$, we have
\[e^{-ic_1 (\frac{\xi}{\varepsilon})^2/2} S_g^\alpha f(\varepsilon^2 x,\frac{\xi}{\varepsilon})\sim C_\alpha\sqrt{\xi}{\varepsilon^{m}}L(\varepsilon)W_g (M_{-\xi c_2}u)(0, \frac{1}{\xi}) \ \ \text{ as } \ \ \varepsilon\rightarrow0^+ \]
$\text{ in } \ \mathcal{S}_0'(\R\times\R^+)$ whenever $\alpha\in(2k\pi,(2k+1)\pi), k\in\mathbb Z$.
\end{theorem}
\begin{proof} For $x\in\R$ and $\xi\in\R^+$, we have
\begin{align*} 
e^{-ic_1(\frac{\xi}{\varepsilon})^2/2}S^\alpha_gf(\varepsilon^2 x,\frac{\xi}{\varepsilon})&=\frac{|\xi|}{\varepsilon} C_\alpha\int_{\R} f(t)\overline g (\frac{\xi}{\varepsilon}(t-\varepsilon^2 x))e^{ic_1 t^2/2-ic_2 t\xi/\varepsilon}dt\\
&=|\xi| C_\alpha\int_{\R} e^{ic_1 (\varepsilon t)^2/2}f(\varepsilon t)\overline g ({\xi}(t-\varepsilon x))e^{-ic_2 t\xi}dt\\
&=|\xi| C_\alpha\langle e^{-ic_2 t\xi}e^{ic_1 (\varepsilon t)^2/2}f(\varepsilon t), \overline g ({\xi}(t-\varepsilon x))\rangle
\end{align*}
and
\begin{align*}
\lim_{\varepsilon\rightarrow0^+} \frac{e^{-ic_1(\frac{\xi}{\varepsilon})^2/2}S^\alpha_gf(\varepsilon^2 x,\frac{\xi}{\varepsilon})}{\varepsilon^m L(\varepsilon)}&=|\xi| C_\alpha\lim_{\varepsilon\rightarrow0^+} \frac{\langle e^{-ic_2 t\xi}e^{ic_1 (\varepsilon t)^2/2}f(\varepsilon t), \overline g ({\xi}(t-\varepsilon x))\rangle}{\varepsilon^m L(\varepsilon)}\\
&=|\xi| C_\alpha\lim_{(\varepsilon_1,\varepsilon_2)\rightarrow(0^+,0^+)} \langle\frac{ e^{-ic_2 t\xi}e^{ic_1 (\varepsilon_1 t)^2/2}f(\varepsilon t)}{\varepsilon_1^m L(\varepsilon_1)}, \overline g ({\xi}(t-\varepsilon_2 x))\rangle.
\end{align*}
Since the weak and the strong topology on $\S_0(\R)$ are equivalent, using \eqref{limL1} and Lemma \ref{l1}, we have
$$\lim_{\varepsilon_1\rightarrow0^+} \langle\frac{ e^{-ic_2 t\xi}e^{ic_1 (\varepsilon_1 t)^2/2}f(\varepsilon t)}{\varepsilon_1^m L(\varepsilon_1)}, \overline g ({\xi}(t-\varepsilon_2 x))\rangle=\langle  e^{-ic_2 t\xi}u(t), \overline g ({\xi}(t-\varepsilon_2 x))\rangle$$
uniformly for $\varepsilon_2\in (0,1]$. Furthermore, for each $\varepsilon_1\in(0,1]$, we have
$$\lim_{\varepsilon_2\rightarrow0^+} \langle\frac{ e^{-ic_2 t\xi}e^{ic_1 (\varepsilon_1 t)^2/2}f(\varepsilon t)}{\varepsilon_1^m L(\varepsilon_1)}, \overline g ({\xi}(t-\varepsilon_2 x))\rangle=\langle\frac{ e^{-ic_2 t\xi}e^{ic_1 (\varepsilon_1 t)^2/2}f(\varepsilon t)}{\varepsilon_1^m L(\varepsilon_1)}, \overline g ({\xi}t)\rangle.$$
Thus
\begin{align*}
\lim_{\varepsilon\rightarrow0^+} \frac{e^{-ic_1(\varepsilon\xi)^2/2}S^\alpha_gf(\varepsilon^2 x,\frac{\xi}{\varepsilon})}{\varepsilon^m L(\varepsilon)}&=|\xi| C_\alpha\langle e^{-ic_2 t\xi}u(t), \overline g ({\xi}t)\rangle\\
&=|\xi| C_\alpha\langle M_{-c_2\xi}u(t), \overline g ({\xi}t)\rangle=C_\alpha\sqrt{\xi}W_g(M_{-c_2\xi} u)(0,\frac{1}{\xi}).
\end{align*}
\end{proof}

\section*{Conflict of interest}
The author declares no conflict of interest.

\end{document}